\documentclass{amsart}
\usepackage{graphicx}
\usepackage{amscd}
\usepackage{amsmath}
\usepackage{amsfonts}
\usepackage{amssymb}
\theoremstyle{plain}
\newtheorem{theorem}{Theorem}

\newtheorem{lemma}[theorem]{Lemma}
\newtheorem{question}[theorem]{Question}
\newtheorem{proposition}[theorem]{Proposition}
\theoremstyle{definition}

\newtheorem{remark}[theorem]{Remark}
\theoremstyle{remark}

\begin{document}
\title{A locally F-finite Noetherian domain that is not F-finite}

\author{Tiberiu Dumitrescu  and Cristodor Ionescu }

\address{Facultatea de Matematica si Informatica,University of Bucharest,14 A\-ca\-de\-mi\-ei Str., Bucharest, RO 010014,Romania}\email{tiberiu\_dumitrescu2003@yahoo.com, tiberiu@fmi.unibuc.ro, }

\address{Simion Stoilow Institute of Mathematics of the Romanian Academy, P. O. Box 1-764, RO-014700 Bucharest, Romania}\email{cristodor.ionescu@imar.ro, cristodor.ionescu@gmail.com}

\thanks{2000 Mathematics Subject Classification: Primary 13A35, Secondary 13F40.}
\thanks{Key words and phrases: Frobenius morphism, F-finite rings, excellent rings.}

\begin{abstract}\noindent Using an old example of Nagata, we construct a Noetherian ring of prime characteristic $p,$ whose Frobenius morphism is  locally finite, but not finite.
\end{abstract}

\maketitle

\section{Introduction}
 Let $p>0$ be a  prime number. We assume that all rings have characteristic $p,$ that is they contain a field of characteristic $p.$ For such a ring $A,$ one can define the Frobenius morphism: 
$$F_A:A\to A,\ F_A(a)=a^p,\ a\in A.$$  
The Frobenius morphism is  playing a major role in studying the properties of $A.$ In \cite[Th. 2.1]{K}  Kunz proved   that a Noetherian local ring of characteristic $p$ is regular if and only if the Frobenius morphism of $A$ is flat. This was the starting point in the study of singularities in characteristic $p$. An important property of Frobenius is its finiteness. A ring $A$ is called \textit{F-finite} if the Frobenius morphism of $A$ is a finite morphism, that is $A$ is a finite $A$-module via $F$. Rings that are $F$-finite  have important properties. For example such rings are excellent, as it is proved by Kunz \cite[Th. 2.5]{K2}. Conversely, a reduced Noetherian ring with $F$-finite total quotient ring  is excellent if and only if it is $F$-finite  \cite[Cor. 2.6]{DS}. Recall that an \textit{excellent ring} is a Noetherian ring $A$ such that (see \cite[Def. p. 260]{Mat}):
\par\noindent\quad i) the formal fibers of $A$ are geometrically regular;
\par\noindent\quad ii) any $A$-algebra of finite type has open regular locus;
\par\noindent\quad iii) $A$ is universally catenary.
\par\noindent Excellent rings were invented by Grothendieck, in order to avoid pathologies in the behaviour of Noetherian local rings.
\par\noindent In \cite{DM} Datta and Murayama asked the following: 
\begin{question}\label{prob}Let $A$ be a Noetherian domain of prime characteristic $p>0.$ Suppose that for any prime ideal $\mathfrak{p}$ of $A$, the localization $A_\mathfrak{p}$ is F-finite. Does it follow that $A$ is F-finite?
 \end{question}

\par\noindent We  show that a certain specialization of Nagata's example  \cite[\S 5]{N}, gives a negative answer to the above question for any prime $p.$ 
\par\noindent All the rings will be commutative and  unitary. Moreover, we fix a prime number $p>0.$

\section{The example}
Assume first that $p$ is odd. Let  $K$ be an algebraically closed field of characteristic $p$ and $X$ an indeterminate. 
Consider the ring
$$
A:=K[X,\frac{1}{\sqrt{(X+a)^3}+\sqrt{b^3}};\ a,b\in K,b\neq 0].
$$
For each square radical above, we 
{\em choose} one of its   two values. Hence the denominator $\sqrt{(X+a)^3}-\sqrt{b^3}$ is not in our list.

\begin{remark}\label{(1)} i) Note that 
$$
\sqrt{(X+a)^3}=\frac{1}{\sqrt{(X+a)^3}+\sqrt{b^3}}((X+a)^3-b^3)+\sqrt{b^3}\in A.
$$

ii) It follows from i) that $A$ is a fraction ring of $K[X,\sqrt{(X+a)^3};\ a\in K]$, which in turn is an integral extension
 of $K[X]$.

iii)  By ii) $A$ is at most one dimensional.
\end{remark}

\begin{lemma}\label{factor}
 The factor ring $A/(X,\sqrt{X^3})$ is isomorphic to $K$, while the factor ring  $A/(X)$ is isomorphic to $K[Y]/(Y^2)$, where $Y$ is an indeterminate.
\end{lemma}

\begin{proof}
From Kneser's Theorem \cite[Th. 5.1]{Kar} we  obtain that  
$$
\sqrt{X+c}\ \notin \  K(X, \sqrt{X+a},\ a\in K,  a\neq c)
$$
for every $c\in K$ (this also follows by  adapting the well-known argument for $\sqrt{p_n}\notin\mathbb{Q}(\sqrt{p_1},\ldots,\sqrt{p_{n-1}}),$ when $p_1,\ldots,p_n$ are distinct primes).
Consequently, we get a ring isomorphism
$$
\frac{K[X, T_a,\ a\in K]}{(T_a^2-(X+a)^3, a\in K)}
\simeq K[X, \sqrt{(X+a)^3},\ a\in K]
$$
sending each indeterminate $T_a$ into $\sqrt{(X+a)^3}$, which extends to an isomorphism
$$
\frac{K[X, T_a,(T_a+\sqrt{b^3})^{-1}, \ a,b\in K,b\neq 0]}{(T_a^2-(X+a)^3, a\in K)}
\simeq A.
$$
It follows that
$$
A/(X)\simeq \frac{K[T_a,(T_a+\sqrt{b^3})^{-1}, \ a,b\in K,b\neq 0]}{(T_a^2-a^3, a\in K)}
$$
$$
\simeq 
\frac{K[T_a,(T_a+\sqrt{b^3})^{-1}, \ a,b\in K,b\neq 0]}{(T_0^2,T_a-\sqrt{a^3}, a\in K,a\neq 0)}
$$
$$
\simeq 
\frac{K[T_0,(\sqrt{a^3}+\sqrt{b^3})^{-1},(T_0+\sqrt{b^3})^{-1}, \ a,b\in K-\{0\}]}{(T_0^2)}\simeq \frac{K[T_0]}{(T_0^2)}
$$
so $A/(X,\sqrt{X^3})$ is isomorphic to $K$.
Note that $\sqrt{a^3}+\sqrt{b^3}$   is nonzero for $a,b\in K-\{0\}$, by our initial one-value-choice for  $\sqrt{b^3}$. Also note that $ T_0+\sqrt{b^3}$ is a unit modulo $ T_0^2$. 
\end{proof}

\begin{lemma}\label{spec} The nonzero prime ideals of $A$  are $(X+a,\sqrt{(X+a)^3})A$ with $a\in K$. In particular, $A$ is a  Noetherian domain of dimension one.
\end{lemma}
\begin{proof} By Lemma \ref{factor}, $ (X,\sqrt{X^3})$ is the only prime ideal of $A$ containing $X$.
By part ii) of Remark \ref{(1)}, every nonzero prime ideal of $A$ contains some $X + a$, so our assertion follows from Lemma \ref{factor},  performing a 
translation in $K$. The final assertion follows from Cohen's Theorem \cite[Th. 3.4]{Mat}.
\end{proof}

\begin{proposition}\label{locffin}
  $A_\mathfrak{m}$ is F-finite for each maximal ideal $\mathfrak{m}$ of $A$.
\end{proposition}
\begin{proof}
Performing a translation in $K$, it suffices to work with $\mathfrak{m}=(X,\sqrt{X^3})$. 
As noted in Remark \ref{(1)}, ii), $A$ is a fraction ring of 
$$
 B:=K[X,\sqrt{(X+a)^3};\ a\in K].
$$
For $a\in K-\{0\}$, we have that $X+a$ is a unit of $A_\mathfrak{m}$ so
$$
\sqrt{(X+a)^3}=\frac{(\sqrt{(X+a)^3})^p}{(X+a)^{pk}}(X+a)^{(p-1)k}\in (A_\mathfrak{m})^p[X,\sqrt{X^3}]
$$
where $k$ is the integer $3(p-1)/2$.
Hence
$$
B\subseteq (A_\mathfrak{m})^p[X,\sqrt{X^3}]\subseteq A_\mathfrak{m}
$$
Let $\mathfrak{n}=\mathfrak{m}A_\mathfrak{m}\cap B$. Since $A$ is a fraction ring of $B$, we get
$$
A_\mathfrak{m}=B_\mathfrak{n}\subseteq (A_\mathfrak{m})^p[X,\sqrt{X^3}],\ \ \mbox{\ therefore\ \ } A_\mathfrak{m}= (A_\mathfrak{m})^p[X,\sqrt{X^3}].
$$\end{proof}

\begin{lemma}\label{reg} The regular locus of $A$ is $\{(0)\}$. 
\end{lemma}
\begin{proof}
Suppose that $A_{\mathfrak{m}_0}$ is  regular for some maximal ideal $\mathfrak{m}_0$ of $A$.
By our translation argument, it follows that $A_\mathfrak{m}$ is  regular for each maximal ideal $\mathfrak{m}$ of $A$. Then  $A$ is normal, so $\sqrt{X}\in $A.  Hence $X$ divides $\sqrt{X^3}$ in $A$, which  contradicts Lemma \ref{factor}.
\end{proof}
\begin{proposition}\label{ffinit}
  $A$ is not F-finite.
\end{proposition}
\begin{proof}
Appy Lemma \ref{reg} and \cite[Cor. 2.3]{K}.
\end{proof}

From Propositions \ref{locffin} and  \ref{ffinit} we get:

\begin{theorem}\label{exemplu}
A is a one-dimensional Noetherian domain that is not F-finite, such that $A_\mathfrak{p}$ is F-finite for any prime ideal $\mathfrak{p}$ of A.
\end{theorem}

\begin{remark}\label{car2}
In characteristic two, a similar  example can be constructed as follows.
Let  $K$ be an algebraically closed field of characteristic two and $X$ an indeterminate. 
Consider the ring
$$
B:=K[X, \sqrt[3]{(X+a)^4}, \frac{1}{\sqrt[3]{(X+a)^8}+
\sqrt[3]{(X+a)^4}\sqrt[3]{b^4} +\sqrt[3]{b^8}};\ a,b\in K,b\neq 0].
$$
For each cube radical above, we 
{\em choose} one of its   three values. 
By Kneser's Theorem \cite[Th. 5.1]{Kar}, it follows that
$$ [K(X,\sqrt[3]{X+a_1},..., \sqrt[3]{X+a_n}):K(X)] = 3^n
$$ 
whenever $a_1$,...,$a_n$ are distinct elements of $K$, so we get a ring isomorphism 
$$
\frac{K[X, T_a, (T_a^2+ T_a\sqrt[3]{b^4} +\sqrt[3]{b^8})^{-1};\ a,b\in K,b\neq 0]}{(T_a^3 - (X+a)^4, \ a\in K)}\simeq B
$$
sending each indeterminate $T_a$ into $\sqrt[3]{(X+a)^4}$.
It easily follows that $B/(X)\simeq K[T_0]/(T_0^3)$ and $B/(X,\sqrt[3]{X^4})\simeq K$. Now all arguments used above can be  adapted to show that $B$ is locally $F$-finite but not $F$-finite.

\end{remark}

{\bf Acknowledgement.} We thank Monica Lewis for pointing us out a misleading argument in a previous version on this paper.

 \end{document}